\documentclass[11pt]{article}

\usepackage{graphicx,epsfig,psfrag}
\usepackage{amsmath,amsfonts,amssymb,amsthm}
\usepackage{a4wide}
\usepackage{rotate}
\usepackage{amsfonts}
\usepackage[english]{babel}
\usepackage{setspace}
\usepackage{fancyhdr,latexsym,enumerate}
\usepackage[latin1]{inputenc}%\documentclass[twocolumn]{autart}
\usepackage{hyperref}
\usepackage{xcolor}

\def\Re{\mathop{\rm Re}}

\def\det{\mathop{\rm det}\nolimits}

\def\CC{{\mathbb C}}
\def\RR{{\mathbb R}}

\def\beq{\begin{equation}}
\def\eeq{\end{equation}}

\newtheorem{proposition}{{\bf Proposition}}[section]

\newtheorem{theorem}{{\bf Theorem}}[section]
\newtheorem{remark}{{\bf Remark}}[section]

\definecolor{Red}{cmyk}{0,1.,1.,0} % PANTONE RED
\newcommand{\red}[1]{#1} % the simplest way to remove all the red bits!

\title{$L_2$ and BIBO stability of systems with variable delays}

\author{Catherine Bonnet\thanks{Universit\'e Paris Saclay, Inria, CNRS, CentraleSupelec, Laboratoire des signaux et syst\`emes, 3 rue Joliot Curie, 91190 Gif-sur-Yvette, France. {\tt
Catherine.Bonnet@inria.fr}}\quad and 
Jonathan R. Partington\thanks{School of Mathematics, University of Leeds, Leeds LS2 9JT, U.K.
\tt{J.R.Partington@leeds.ac.uk}}}

\begin{document}

\maketitle

\begin{abstract}
 This paper considers  $L_2$ and BIBO stability and stabilization issues for systems  with time-varying delays which can be of retarded or neutral type. An important role is played by a nominal system with fixed delays which are close to the time-varying ones. Under stability  or stabilizability conditions of this nominal system,  sufficient conditions are given in order to ensure  similar properties for the system with time-varying delays. 
\end{abstract}


\vskip0,5cm

\centerline{\bf In memory of Ruth Curtain.}

\section{Introduction and notation}

%Autonomous system with time-varying delays have been largely studied in the recent years however much 
% fewer works have been dedicated to such systems submitted  to external inputs. 

There has been a growing interest in systems with time-varying delays in recent years, in particular due to the study of 
large-scale networked systems where delays may have to be considered to be time-varying for better accuracy of the model. The presence of time-varying delays induces complex behaviours which have been pointed out and studied by many authors (see e.g. \cite{L01, MVN05, V12} the monograph \cite{F14} and the references therein).
Although the case of autonomous systems has been widely studied, much less work has been dedicated to such systems
submitted to arbitrary external inputs. \red{Several studies, e.g. \cite{B15}, consider the $L^2$-input-output norm of the time-varying delay operator}, however to the best of our knowledge, no study has been made of the input--output stability of
a standard feedback scheme involving a system with time-varying delays. We can find in \cite{quet} an input-output setting for determining rate-based flow controllers for communication networks in the presence of time-varying delays; however, an $H_\infty$ optimization problem is considered here which deals with several types of perturbations acting on the state-space system including time-varying coefficients and time-varying delays and this renders the comparison with our work inappropriate. 
Several papers (for example \cite{F08,FG,KR07,SF07}) mention an
input--output approach for systems with time-varying delays; 
however the input--output framework is not considered for the given
system but is introduced once an initially autonomous 
system with time-varying delays has been transformed into a system
with external input by means of the introduction of additional
variables given in terms of the states. 

In our setting, the classical
condition 
on the boundedness of the derivative of the time-varying delay by a
constant strictly lesser than one is not necessary. Note that this is
also the case in \cite{yw16} where $H_\infty$ state-feedback control
of systems time-varying input delays is considered. Physical
considerations (such as the order of events being preserved in the presence of a
time-varying delay) may impose that for a time-varying delay $\tau(t)$
the function $t \mapsto t -\tau(t)$ is strictly increasing, but since this only
implies that the derivative of $\tau$ is bounded by one, a relaxation of the
classical condition of boundedness is of use ({\em moderately varying delays} were first considered in \cite{FS}). 

We consider two versions of input--output stability: so-called $H_\infty$-stability (a finite gain when
the signals are measured in terms of energy, that is, $L_2$ norm), and BIBO stability (bounded inputs
give bounded outputs). Retarded and neutral type systems with pointwise or distributed delays are studied here. \red{Note that in \cite{F08}, the input-output approach to neutral systems has been presented, but the results were confined to slowly-varying delays.}

The very general framework is introduced in Section~\ref{sect-def-syst}, and the stability analysis is
given in Section~\ref{sect-stab}. Then, a discussion of feedback stabilization is given in Section~\ref{sec:stabilization}.
Finally, numerical examples appear in Section~\ref{sec:ex}.
The results we derive yield very explicit estimates for stability margins, as will be seen by
considering the examples we present.  

\subsection*{Notation}

$H_\infty$, or $H_\infty(\CC_+)$, denotes the space of bounded analytic functions on the complex
right half-plane $\CC_+$, with the supremum norm
$\|F\|_\infty = \sup _{s \in \CC_+} |F(s)|$. The $H^\infty$ norm gives an explicit expression for the $L_2$ input/output gain of a stable system.

In addition, ${\mathcal A}$ denotes the space of distributions of the form 
\[
g(t)=g_a(t) + \sum_{i=0}^\infty
g_i \delta(t-t_i),
\] where $g_a \in L_1(0,\infty)$ and $\sum_0^\infty |g_i|<\infty$; here,
$\delta(t-t_i)$ denotes a delay Dirac distribution, \red{$0 = t_0 <t_1, t_2,t_3 \ldots  $ the $t_k$ being distinct.} 
This is equipped with the norm 
\[
\|g\|_{\cal A}= \int_0^\infty |g_a(t)| \, dt + \sum_{i=0}^\infty |g_i|,
\]
which gives an expression for the $L_\infty$ input/output gain of a stable system.
Then $\hat{\cal A}$
is the space of Laplace transforms of functions in $\cal A$, which is a subalgebra of $H^\infty$.

\section{The class of systems studied}
\label{sect-def-syst}

In this paper we shall consider the variable-delay input--output system
\begin{eqnarray}
\dot x(t) + \sum_{\ell=1}^L A_{-\ell}\dot x(t-\gamma_\ell(t))&=& Ax(t)+\sum_{j=1}^J A_j x(t-\tau_j(t))  + \int_0^{\delta(t)} h(\theta) I x(t-\theta) \, d\theta \nonumber\\
&&
 + Bu(t)+ \sum_{k=1}^K B_k u(t-\sigma_k(t))  \, d\theta, \quad (t > 0),
\label{eq:system1}
\end{eqnarray}
where  $x(t) \in \CC^n$ denotes the state and $u(t) \in \CC^p$ the input, both assumed zero for $t\le 0$, the matrices
$A$, $A_{-\ell}$, $A_j$, $B$, $B_k$ and $I$ (identity matrix) have the appropriate sizes, the function $h(\theta)$ is of polynomial type and the delays $\gamma_\ell(t)$, $\tau_j(t)$, $\sigma_k(t)$ and $\delta(t)$ are
positive (to ensure stability, further hypotheses will be required later).\\

Were the delays constant (say, $\gamma_\ell(t)=H_\ell$, $\tau_j(t)=h_j$,
$\delta(t)=D$,
 and $\sigma_k(t)=T_k$ for each $j$, $k$ and $l$), then this system would have a transfer function given by
\begin{equation*}\label{eq:transfer1}
G(s) = \left (sI + \sum_{\ell=1}^L A_{-\ell}se^{- H_\ell s}-A-\sum_{j=1}^J A_j e^{-h_j s}-\int_0^D h(\theta)e^{-s\theta} \, d\theta \right)^{-1} 
\left (B+ \sum_{k=1}^K B_k e^{-T_k s} \right),
\end{equation*}
and $H_\infty$ stability (bounded $L_2$ gain) would correspond to $G$ being analytic and bounded
in the right half-plane $\CC_+$. \\

Let us mention that the considered distributed delay will give rise only to polynomial and exponential terms in the transfer function allowing one to divide the considered class of systems into retarded and neutral type classes as in the generic (distributed delay free) case. \\

% In fact, as we shall see later, we shall also need to assume that $sG(s) \in H_\infty$.

In \cite{FG}, Fridman and Gil' considered an autonomous system with variable delays, given by
\[
\dot x(t) = \sum_{j=1}^\ell A_j x(t-\tau_j(t)) , \qquad (t > 0),
\]
with initial conditions specified on an interval $[-\eta, 0]$. They analysed its asymptotic
stability using a frequency-domain approach. 

It is our aim to analyse the $H_\infty$ and BIBO stability of the system (\ref{eq:system1}), using
methods similar to those of \cite{FG} in a far more general context, assuming stability of the nominal system where the delays are fixed.\\

Consider the {\em nominal system\/} (as always, with zero initial conditions),
\beq\label{eq:system2}
\dot v(t) +\sum_{\ell=1}^L A_{-\ell} \dot v(t-H_\ell) = Av(t)+\sum_{j=1}^J A_j v(t-h_j) +\int_0^{D} h(\theta) I v(t-\theta)\, d\theta+ Bu(t)+ \sum_{k=1}^K B_k u(t-T_k),  
\eeq 
for $t>0$,

\medskip

Throughout the rest of the paper, we make the following hypothesis.

\medskip

\noindent {\bf Hypothesis (H):} we assume that $\sum_{\ell=1}^L \Vert A_{-\ell} \Vert <1$.

%\noindent {\bf Hypothesis 2:} the delays $H_\ell$ ($\ell=1, \cdots, L$) are supposed commensurate to a delay $H$ ($H_\ell = \ell H$).
% 
 
\medskip

\medskip
If  the $L_2$ gain from $u$ to $v$  is finite we have that $\dot v(t) +\sum_{\ell=1}^L A_{-\ell} \dot v(t-H_\ell)$ is in $L_\infty$ and there exists a constant K such that 
\[
\left\| \dot v(t) +\sum_{\ell=1}^L A_{-\ell} \dot v(t-H_\ell)\right\|_\infty  \leq K \Vert u \Vert _\infty.
\]
  From this inequality and the fact that $\Vert \dot v(t-H) \Vert_\infty  = \Vert \dot v(t) \Vert _\infty$ , we get that 
  \[
 \left (1-\sum_{\ell=1}^L \| A_{-\ell}\| \right) \Vert \dot v \Vert \leq K \Vert u \Vert_\infty,
  \]
   which means, as  $\sum_{\ell=1}^L \Vert A_{-\ell} \Vert <1$, that we also have a finite $L_\infty$  gain  from $u$ to $\dot v$.\\

If  the $L_2$ gain from $u$ to $v$  is finite we have that the associated transfer function $G_{nom}$ is in $H_\infty$, its denominator being 
\[
%\det(sI +\sum_{\ell=1}^L A_{-\ell}s e^{- H_\ell s}-A-\sum_{j=1}^J A_j e^{-h_j s}-\int_0^D h(\theta)e^{-\theta s})^{-1}.
\det(sI +\sum_{\ell=1}^L A_{-\ell}s e^{- H_\ell s}-A-\sum_{j=1}^J A_j e^{-h_j s}-\int_0^D h(\theta) I e^{-\theta s}  d\theta)
\]
 We can easily deduce from this that $sG(s)$ is also in $H_\infty$ as there is no problem with properness and no problem with boundedness on the imaginary axis. This means that there is a   finite  $L_2$ gain from $u$ to $\dot v$.

We shall need to consider the system 
\[
\dot z(t) + \sum_{\ell=1}^L A_{-\ell} \dot z(t-H_\ell) = Az(t) +\sum_{j=1}^J A_j z(t-h_j) + \int_0^D h(\theta) I z(t-\theta) \, d\theta +
w(t), \qquad (t > 0).
\]

For the same reasons as above, if  the $L_2$-gain  (from $w$ to $z$)  is finite, we also have a  $L_2$  finite gain from $w$ to $\dot z$ and  if  the $L_\infty$-gain  (from $w$ to $z$)  is finite we also have a finite $L_\infty$ gain from $ w$ to $\dot z$.

\medskip

\noindent Let us denote :

\medskip

\begin{tabular}{lll}
$M^{nom}_2$ & : &  the $L_2$ gain from $u$ to $v$ \\
$M^{nom}_\infty$  & : &  the $L_\infty$ gain from $u$ to $v$\\
$M^{nomd}_\infty$  & : & the $L_\infty$  gain  from $u$ to $\dot v$ \\
$M^{nomd}_2$ &:& the $L_2$ gain from $u$ to $\dot v$ \\ 
$M_2$ &:& the $L_2$ gain  from $w$ to $z$ \\
$M_\infty$ &:& the $L_\infty$ gains from $w$ to $z$\\
$M_2^d$ &:& the $L_2$  gain from $w$ to $\dot z$\\
 $M_\infty^d$ &:& the $L_\infty$ gain from $ w$ to $\dot z$ 
\end{tabular} 

\section{Stability analysis}
\label{sect-stab}
Our standing assumptions, together with Hypothesis (H) are that 
\medskip

$\begin{array}{l}
0 \le  H_\ell-\eta_\ell \le \gamma_\ell(t) \le H_\ell + \eta_\ell \\
0 \le h_j-\mu_j \le \tau_j(t) \le h_j + \mu_j \\ 
0 \le D-\epsilon \le \delta(t) \le D+\epsilon\\
0 \le T_k-\nu_k \le \sigma_k(t) \le T_k + \nu_k 
\end{array} $

\medskip

for all $t$ and for each $j$, $k$ and $\ell$, where $\eta_\ell, \mu_j, \epsilon, and  \nu_k$  are positive constants.

\medskip

Writing $y=x-v$ we have
\begin{eqnarray*}
\dot y(t) +  \sum_{\ell=1}^L A_{-\ell }\dot y(t-H_\ell) & = & A y(t) \\
& & + \sum_{j=1}^J A_j y(t-h_j)
+ \sum_{j=1}^J A_j[ x(t-\tau_j(t))-x(t-h_j)]\\  
& & + \int_0^D h(\theta)I  y(t-\theta) \, d\theta + \int_D^{\delta(t)} h(\theta) I x(t-\theta) \, d\theta \\
&  & +\sum_{\ell=1}^L A_{-\ell}[\dot x(t-H_\ell) - \dot x(t-\gamma_\ell(t))]\\
& &  + \sum_{k=1}^K B_k [ u(t-\sigma_k(t)) - u(t-T_k)], \qquad (t>0).
\end{eqnarray*}
We may write this as
\begin{eqnarray}
\dot y(t)  +    \sum_{\ell=1}^L A_{-\ell }\dot y(t-H_\ell)& = & A y(t) + \sum_{j=1}^J A_j y(t-h_j)  + \int_0^D h(\theta)I  y(t-\theta) \, d\theta\nonumber\\
& & + \sum_{j=1}^J A_j f_j(t) + \sum_{\ell=1}^L A_{-\ell }\tilde{f}_\ell(t) 
+ \int_D^{\delta(t)} h(\theta) I x(t-\theta) \, d\theta
+ \sum_{k=1}^K B_k g_k(t) \nonumber\\ && \label{eq:system3}
\end{eqnarray}
with 
\begin{eqnarray*}
f_j(t) &=&  x(t-\tau_j(t))-x(t-h_j), \\
\tilde{f}_\ell(t) &=& \dot x(t-H_\ell) - \dot x(t-\gamma_\ell(t)), \\
\hbox{and} \qquad g_k(t) &=& u(t-\sigma_k(t)) - u(t-T_k).
\end{eqnarray*}

We now need supplementary conditions to ensure that the functions $\tilde{f}$, $f_j$ and $g_k$ lie in $L_2(0,\infty)$ or $L_\infty(0,\infty)$.

\subsection{The special case of retarded type systems}

We begin with the case that
 all $A_{-\ell}=0$, since the system is then more robust to perturbations. 
 The equations under consideration are 
\beq\label{eq:system1-ret}
\dot x(t) = Ax(t)+\sum_{j=1}^J A_j x(t-\tau_j(t))  + \int_0^{\delta(t)} h(\theta) I x(t-\theta) \, d\theta +
Bu(t)+ \sum_{k=1}^K B_k u(t-\sigma_k(t))    \quad (t > 0),
% \dot x(t) = Ax(t)+\sum_{j=1}^\ell A_j x(t-\tau_j(t)) + Bu(t)+ \sum_{k=1}^m B_k u(t-\sigma_k(t)), \qquad (t > 0),
\eeq

\beq\label{eq:system2-ret}
% \dot v(t) = Av(t)+\sum_{j=1}^\ell A_j v(t-h_j) + Bu(t)+ \sum_{k=1}^m B_k u(t-T_k), \qquad (t > 0),
\dot v(t) = Av(t)+\sum_{j=1}^J A_j v(t-h_j) +\int_0^{D} h(\theta) I v(t-\theta)\, d\theta+ Bu(t)+ \sum_{k=1}^K B_k u(t-T_k), \qquad (t>0)
\eeq 

\beq\label{eq:system3-ret}
% \dot v(t) = Av(t)+\sum_{j=1}^J A_j v(t-h_j) +\int_0^{D} h(\theta) v(t-\theta)\, d\theta+ Bu(t)+ \sum_{k=1}^K B_k u(t-T_k), 
 \dot y(t) = A y(t) + \sum_{j=1}^J A_j y(t-h_j) + \int_0^D h(\theta) I y(t-\theta) \, d\theta+ \sum_{j=1}^J A_j f_j(t) + \int_D^{\delta(t)} h(\theta) I x(t-\theta) \, d\theta
+ \sum_{k=1}^K B_k g_k(t) , 
\eeq
and
\beq\label{eq:system4-ret}
\dot z(t)  = Az(t) +\sum_{j=1}^J A_j z(t-h_j) + \int_0^D h(\theta) I  z(t-\theta) \, d\theta +
w(t), \qquad (t > 0).
\eeq

\begin{theorem}\label{thm:2.1}
\noindent 1) Suppose that the system (\ref{eq:system2-ret})  is BIBO-stable. 

\noindent If $ M_\infty^d \sum_{j=1}^\ell \mu_j \|A_j\|<1$  and $ M_\infty \epsilon \|h \|_\infty  \left(\sum_{j=1}^\ell M_\infty^d \mu_j\|A_j\| (1-M)^{-1}   +1 \right) <1 $, then the system (\ref{eq:system1-ret}) is BIBO-stable.\\

\noindent 2) Suppose that the system (\ref{eq:system2-ret})  is $H_\infty$-stable and that $\dot u \in L_2$. 

\noindent If $ M=M_2^d \sum_{j=1}^\ell \mu_j \|A_j\|<1$  and $M_2 \epsilon \|h \|_\infty  ( \sum_{j=1}^\ell \mu_j\|A_j\| (1-M')^{-1}  +1) <1 $, then the system (\ref{eq:system1-ret}) is $H_\infty$-stable in the sense that there is a finite $L^2$ gain between $(u, \dot u)$ and $x$.
\end{theorem}

\begin{proof}

1) The basic calculation is as follows:
\begin{eqnarray*}
\|\dot x\|_\infty &\le& \|\dot v\|_\infty + \|\dot y\|_\infty \\ 
& \le & \|\dot v\|_\infty +   M^{d}_\infty (\sum_{j=1}^\ell \|A_j\|\|f_j\|_\infty + \epsilon \Vert h\Vert_\infty \Vert x \Vert_\infty + \sum_{k=1}^m \|B_k\|\|g_k\|_\infty)
\end{eqnarray*}
and
then we shall bound $\|f_j\|_\infty$ in terms of $\|\dot x\|_\infty$. 

We have 
\[
 \|f_j\|_\infty =  \sup_t \left \| \int_{t-\tau_j(t)}^{t-h_j} \dot x(s) \, ds \right\| 
 \le  \mu_j \|\dot x\|_\infty
\]

and

\[
 \|g_k\|_\infty =  \sup_t  \Vert  [ u(t-\sigma_k(t)) - u(t-T_k)] \vert  \leq  2  \Vert u \Vert _\infty
\]

So that, provided that   $M:=M_\infty^d \sum_{j=1}^\ell \mu_j \|A_j\| < 1$, we have

\[
 \|\dot x\|_ \infty \le \|\dot v\|_\infty  + M  \|\dot x\|_\infty +M_\infty^d \epsilon \Vert h\Vert_\infty \Vert x \Vert_\infty +  M_\infty^d \sum_{k=1}^m 2 \|B_k\|  \|u\|_\infty,
 \]
 or
 \[
 \|\dot x\|_\infty \le (1-M)^{-1}\left( M_\infty^{nomd}\|u\|_\infty + M_\infty^d \epsilon \Vert h\Vert_\infty \Vert x \Vert_\infty  + M_\infty^d \sum_{k=1}^m 2 \|B_k\|  \|u\|_\infty \right),
 \]

% and hence there is a bound on $\|f_j\|_\infty$.

Now $x=v+y$ and so by (\ref{eq:system2-ret}) and (\ref{eq:system3-ret}) we have 
\begin{eqnarray*}
\|x\|_\infty &\le & \|v\|_\infty + \|y\|_\infty \\
& \le &  M_\infty^{nom} \|u\|_\infty + M_\infty (\sum_{j=1}^\ell \|A_j\|\|f_j\|_\infty + \epsilon \Vert h\Vert_\infty \Vert x \Vert_\infty +  \sum_{k=1}^m \|B_k\|\|g_k\|_\infty) \\
& \le &  M_\infty^{nom}  \|u\|_\infty + M_\infty  (\sum_{j=1}^\ell \mu_j\|A_j\|\|\dot x\|_\infty +  \sum_{k=1}^m 2 \|B_k\|\| u\|_\infty + \epsilon \Vert h\Vert_\infty \Vert x \Vert_\infty )
\end{eqnarray*}
Let $ \displaystyle \tilde{M}:= M_\infty \epsilon \Vert h\Vert_\infty \left(1+ \sum_{j=1}^\ell \mu_j\|A_j\|\|M_\infty^d (1-M)^{-1}\right)$, we have
\begin{eqnarray*}
(1-  \tilde{M}) \|x\|_\infty  &\leq & M_\infty^{nom} \|u\|_\infty  + \\
& & M_\infty  \left( \sum_{j=1}^\ell \mu_j\|A_j\| (1-M)^{-1} ( M_\infty^{nomd} + M_\infty^d \sum_{k=1}^m 2 \|B_k\| ) +  \sum_{k=1}^m 2 \|B_k\| \right)  \|u\|_\infty \\
% &  & +  \sum_{k=1}^m 2 \|B_k\| ) ) \|u\|_\infty ,\\
\end{eqnarray*} 
which under the condition  $ \tilde{M} < 1 $ gives a finite $L_\infty$ gain from $u$ to $x$.\\

2) In the case of $L_2$-stability, we start with the same inequality as above.  
First, similarly to \cite{FG}, we have, 
recalling that $x(t)=\dot x(t)=0$ for $t \le 0$,
\begin{eqnarray*}
\|f_j\|^2_2 &=& \int_0^\infty \left \| \int_{t-\tau_j(t)}^{t-h_j} \dot x(s) \, ds \right\|^2 \, dt
\le \int_0^\infty \mu_j \int_{t-h_j-\mu_j}^{t-h_j} \| \dot x(s)\|^2 \, ds \, dt \\
& \le & \mu_j^2 \int_0^\infty \|\dot x(r)\|^2 \, dr = \mu_j^2 \|\dot x\|_2^2.
\end{eqnarray*}

Then, to  bound $\|g_k\|_2$ requires some restrictive
conditions on $u$.  With the condition that $\dot u \in L_2$, we get $\|g_k\|_2 \le \nu_k\|\dot u\|_2$ and the stability result obtained in the case that there are
input delays will take the form
\[
\|x\| _2\le C_1 \|u\|_2 + C_2 \|\dot u\|_2.
\]

Here again, provided that $M':=M_2^d \sum_{j=1}^\ell \mu_j \|A_j\| < 1$, we have

 \[
 \|\dot x\|_2 \le \|\dot v\|_2 + M' \|\dot x\|_2+ \epsilon \|h \|_\infty  \|x \|_2+ M_2^d   \sum_{k=1}^m \nu_k\|B_k\|  \|\dot u\|_2,
 \]
 or
 \[
 \|\dot x\|_2 \le (1-M')^{-1}\left(M_2^{nomd}\|u\|_2 + \epsilon \|h \|_\infty \|x \|_2  + M_2^d \sum_{k=1}^m \nu_k\|B_k\|  \|\dot u\|_2\right),
 \]

 By (\ref{eq:system2-ret}) and (\ref{eq:system3-ret}) we have 
\[ \|x\|_2  \le   M_2^{nom}  \|u\|_2+   M_2     (\sum_{j=1}^\ell \mu_j\|A_j\| \|\dot x\|_2 + \epsilon \|h \|_\infty  \|x \|_2  +  \sum_{k=1}^m \nu_k\|B_k\|\|\dot u\|_2), 
\]
that is,
  
% \red{\begin{eqnarray*}
\begin{multline*}
  \left(1-M_2 \epsilon \|h \|_\infty  (  \sum_{j=1}^\ell \mu_j\|A_j\| (1-M')^{-1}  +1) \right) \|x\|_2   
 \\  \leq   \left( M_2^{nom} + M_2 M_2^{nomd}\sum_{j=1}^\ell \mu_j\|A_j\| (1-M')^{-1}  \right)  \|u\|_2  
+ \left( M_2  \sum_{k=1}^m \nu_k\|B_k\|  ( M_2^d   +1) \right) \| \dot u\|_2,
\end{multline*}
% \end{eqnarray*}}
which gives, if  $ M_2 \epsilon \|h \|_\infty  ( \sum_{j=1}^\ell \mu_j\|A_j\| (1-M')^{-1}  +1) <1$, a finite $L_2$ gain from $(u,\dot u)$ to $x$ if there are input delays.

\end{proof}

\begin{remark}
If $B_k=0$ for all $k$, or if the $\sigma_j$ are constant functions for all $j$, we do not need to impose the condition that $\dot u$ lies
in $L_2$ in order to get $H_\infty$-stability.

\end{remark}
\subsection{The general case of neutral systems}

In general it is easier to destabilize a neutral delay system by a perturbation, and here we simply give a result for BIBO stability.

\begin{theorem}
% 1) 
Suppose that the system (\ref{eq:system2}) is BIBO-stable. 

\noindent If $ M_\infty^d (\sum_{j=1}^\ell \mu_j \|A_j\|  + 2 \sum_{l=1}^L \Vert A_{-l}\Vert)<1$ then the system (\ref{eq:system1}) is BIBO-stable.

% \noindent 2) Suppose that the system (\ref{eq:system2})  is $H_\infty$-stable and that $\dot u \in L_2$. If $ <1$  then the 
% system (\ref{eq:system1}) is $H_\infty$-stable in the sense that there is a finite gain between $(u, \dot u)$ and $x$ ????
\end{theorem}

\begin{proof}
  As for retarded systems we start with the inequality 
\[
\|\dot x\|_\infty \le \|\dot v\|_\infty +   M^{d}_\infty (\sum_{j=1}^\ell \|A_j\|\|f_j\|_\infty +  \sum_{k=1}^m \|B_k\|\|g_k\|_\infty  +\sum_{l=1}^L \Vert A_{-l} \Vert \tilde{f}\Vert _\infty    + \epsilon \Vert h \Vert_\infty \Vert x \Vert _\infty).
\]

We have 
\[
 \|\tilde f\|_\infty =  \sup_t \Vert \dot x(t-\gamma(t) )- \dot x (t-H)  \Vert \leq 2  \Vert \dot x\Vert _\infty,
\]
whereas 
\[
 \|f_j\|_\infty \le  \mu_j \|\dot x\|_\infty
\]
and
\[
 \|g_j\|_\infty  \leq  2  \Vert u \Vert _\infty
\]
as in the retarded case.

Provided that $M'':=M_\infty^d (\sum_{j=1}^\ell \mu_j \|A_j\|  + 2 \sum_{l=1}^L \Vert A_{-l} \Vert) < 1$, we have

 \[
 \|\dot x\|_\infty \le (1-M'')^{-1}\left( M_\infty^{nomd}\|u\|_\infty  + M_\infty^d \sum_{k=1}^m 2 \|B_k\|  \|u\|_\infty   + \epsilon \Vert h \Vert_\infty \Vert x \Vert _\infty \right),
 \]
% and hence there is a bound on $\|f_j\|_\infty$.

From the relation $x=v+y$, we get  a finite $L_\infty$ gain from $u$ to $x$.

% 2) Let us derive each member of the equality  (\ref{eq:system2}). By the same reasoning as in section \ref{sect-def-syst} we get that there is a finite $L_2$ gain between $\dot u $ and $\ddot y$.

% We have
% \begin{eqnarray*}
% \|\tilde f\|^2_2 &=& \int_0^\infty \left \| \int_{t-\gamma(t)}^{t-H} \ddot x(s) \, ds \right\|^2 \, dt
% \le \int_0^\infty \eta \int_{t-H-\eta }^{t-H} \| \ddot x(s)\|^2 \, ds \, dt \\
% & \le & \eta^2 \int_0^\infty \|\ddot x(r)\|^2 \, dr = \eta^2 \|\ddot x\|_2^2.
% \end{eqnarray*}

\end{proof}

\section{Stabilization properties}
\label{sec:stabilization}

\subsection{The case of retarded systems}
We consider here the stabilization of  system  (\ref{eq:system1-ret}) through the following standard feedback scheme where $r$ and $d$ are  external input signals. \\

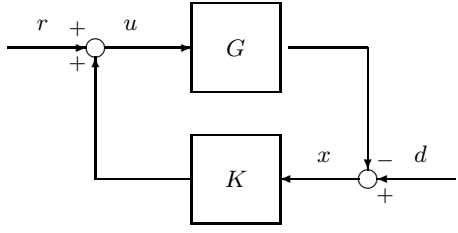
\begin{figure}[hbt]
  \centering
  \hskip3cm
  \footnotesize
  \setlength{\unitlength}{.0075\textwidth}
  \begin{picture}(80,35)
  \put(21,25){\framebox(10,10){$G$}}
  \put(21,10){\framebox(10,10){$K$}}
  \thinlines
  \put(0,30){\vector(1,0){9}}
  \put(4,32){\makebox(0,0)[b]{$r$}}
 \put(14,32){\makebox(0,0)[b]{$u$}}
  \put(10,30){\circle{2}}
\put(9,33){\makebox(0,0)[tr]{\scriptsize $+$}}
\put(9,29){\makebox(0,0)[tr]{\scriptsize $+$}}

  \put(11,30){\vector(1,0){10}}
  \put(32,30){\line(1,0){9}}
  \put(41,30){\vector(0,-1){14}}
  \put(41,15){\circle{2}}
  \put(51,15){\vector(-1,0){9}}
\put(47,17){\makebox(0,0)[b]{$d$}}
\put(44,18){\makebox(0,0)[tr]{\scriptsize $-$}}
\put(44,14){\makebox(0,0)[tr]{\scriptsize $+$}}
\put(40,15){\vector(-1,0){9}}
\put(21,15){\line(-1,0){11}}
\put(10,15){\vector(0,1){14}}
\put(36,17){\makebox(0,0)[b]{$x$}}
  \end{picture}
  \normalsize
  \caption{Standard Feedback Configuration}
  \label{standard}
\end{figure}

Retarded systems are $H_\infty$-stabilizable  (see \cite{BP99}) and so strongly stabilizable (see \cite{quadrat}). They are also $BIBO$-stabilizable (see \cite{BP99}) and those with commensurate delays are strongly $BIBO$-stabilizable (see \cite{MS10}). 

Therefore, we may consider here only strong stabilization of system (\ref{eq:system1-ret}) by a controller $K$, with convolution kernel ${ \cal K } $. 
In this case, instead of studying the stability of the four matrix entries between $(r,d)^ T  $ and $(x,u) ^ T  $, it becomes sufficient to only study the stability of that between $r$ and $x$.  
In the following, the control law $u$ is then taken of the type $u={\cal K}*x +r$.

The closed-loop has the following equation 

\begin{eqnarray}
\dot x(t)  & =& Ax(t)+\sum_{j=1}^\ell A_j  x(t-\tau_j(t)) + \int_0^{\delta(t)} h(\theta) x(t-\theta) \, d\theta   + B ({\cal K}*x)(t) + \sum_{k=1}^m B_k ({\cal K}*x)(t-\sigma_k(t) ) 
\nonumber\\
&  &  +Br(t) + \sum_{k=1}^m B_k r(t-\sigma_k(t) ), \qquad (t > 0),
\label{eq:systemcl}
\end{eqnarray}

Consider the system

\begin{eqnarray}
\label{eq:system10}
\dot v(t)  & =& Av(t)+\sum_{j=1}^\ell A_j v(t-h_j) +\int_0^{D} h(\theta) v(t-\theta)\, d\theta + B ({\cal K}*v)(t) + \sum_{k=1}^m B_k ({\cal K}*v)(t-T_k) 
\nonumber\\
&  &  +Br(t) + \sum_{k=1}^m B_k r(t-T_k), \qquad (t > 0),
\end{eqnarray} 

and let us denote:

\medskip

\begin{tabular}{lll}
 $M^{cl}_\infty$  & : &  the  $L_\infty$-gain of the closed-loop \\
$M^{cl d}_\infty$ &:& $L_\infty$-gain  between input and derivative of the state of the closed-loop\\
$M^{cl}_2$&:& the  $L_2$-gain of the closed-loop\\
$M^{cl nom d}_\infty$ &:& $L_\infty$-gain  between input $r$ and $\dot{v}$
\end{tabular}

\begin{theorem}\label{thm:4.1}
%Let us consider system (\ref{eq:system1-ret}). 

1) Let us suppose that the nominal system (\ref{eq:system2-ret}) has commensurate delays. Let $ K$ be a stable controller which $BIBO$-stabilizes the nominal system (\ref{eq:system2-ret}), the $L_\infty$-gain of $K$ being denoted $M^K_\infty$. 

If  $M^{cl}_\infty(\sum_{j=1}^\ell \mu_j \Vert A_j\Vert  +  \sum_{k=1}^m \nu_k \Vert B_k \Vert  M^  K_\infty  +\epsilon \Vert h \Vert _\infty )<  1$ then the controller $K$ stabilizes system (\ref{eq:system1-ret}) in a BIBO sense.\\

\noindent 2)  Let $K$ be a stable controller which $H_\infty$-stabilizes the nominal system (\ref{eq:system2-ret}), the $L_2$-gain of $ K$ being denoted $M^  K_2$. Let us suppose that $\dot r$ is in $L_2$. 

If  $  M^{cl}_2(\sum_{j=1}^\ell  \mu_j  \Vert A_j\Vert +  \sum_{k=1}^m \nu_k  \Vert B_k \Vert M^  K_2 )< 1$ then the controller $K$ stabilizes system (\ref{eq:system1-ret}) in the sense that there is a finite $L_2$ gain between $r$ and $x$.

\end{theorem}

\begin{proof}

1) Let $\tilde K$ be a stable BIBO-stabilizing controller of system (\ref{eq:system2-ret}) with convolution kernel ${\cal K}$.

Considering the stabilization of (\ref{eq:system1-ret}) and (\ref{eq:system2-ret}) we deal with  the following equations, where as before $y=x-v$:

and

\begin{eqnarray*}
\label{eq:system11}
\dot y(t) & = & Ay(t)+\sum_{j=1}^\ell A_j y(t-h_j)   +\sum_{j=1}^\ell A_j f_j(t) \nonumber\\
&  & + B ({\cal K}*y)(t) + \int_D^{\delta(t)} h(\theta)x(t-\theta) \, d\theta \nonumber  \\
&   &+ \sum_{k=1}^m B_k ({\cal K}*\tilde f _k)(t) +\sum_{k=1}^m B_k ({\cal K}*{\tilde r})(t) , \qquad (t > 0),
\end{eqnarray*}
 with  $f_j(t)=x(t-\tau_j(t))-x(t-h_j)$, $\tilde f _k(t)=x(t-\sigma_k(t))-x(t-T_k)$ and $\tilde r(t)=r(t-\sigma_k(t))-r(t-T_k)$.

Using the same arguments as in section \ref{sect-stab}, we get  
\beq
\Vert \dot x \Vert_\infty \leq \Vert \dot v \Vert_\infty  + M_\infty^{cl d}(\sum_{j=1}^\ell \mu_j \Vert A_j\Vert  \Vert \dot x \Vert_\infty  + \sum_{k=1}^m \nu_k \Vert B_k \Vert   M_\infty ^K  \Vert \dot x\Vert  + 2 \sum_{k=1}^m \Vert B_k \Vert M_\infty ^K \Vert r \Vert_\infty + \epsilon \Vert h \Vert _\infty \Vert x \Vert _\infty).
\nonumber
\eeq

If $1- M_\infty ^{cl d}(\sum_{j=1}^\ell  \mu_j  \Vert A_j\Vert +  \sum_{k=1}^m \nu_k  \Vert B_k \Vert  M_K ) >0$, the above inequality gives us:
% the following  bound  for $\Vert \dot x \Vert $:

\beq
(1- M_\infty^{cld}(\sum_{j=1}^\ell \mu_j  \Vert A_j\Vert  +  \sum_{k=1}^m \nu_k  \Vert B_k \Vert  M_K) ) \Vert \dot x \Vert  \leq ( M_\infty ^{cl nom d}  + 2  M_\infty ^{cl d} \sum_{k=1}^m \Vert B_k \Vert M_\infty ^K )  \Vert r \Vert_\infty +\epsilon \Vert h \Vert _\infty \Vert x \Vert _\infty
\nonumber
\eeq 

%This gives a bound for the $f_j$ and the $\tilde f _k$, and 
As before writing $x=y+v$ we get a $L_\infty$-bound between $r$ and $x$.

2)  Taking $\dot  r $ in $L_2$ enables us, as in the Proof of Theorem  \ref{thm:2.1}, to bound $\Vert \dot  r \Vert_2$, which proves the result.

\end{proof}

\subsection{The case of neutral systems}

\begin{proposition}
%Let us consider system (\ref{eq:system1-ret}). 
%\red{Let us suppose that the nominal system (\ref{eq:system2}) has commensurate delays and satisfies hypothesis 1.}
Let us suppose that there exits a stable controller  $ K$ which $BIBO$-stabilizes the nominal system (\ref{eq:system2}), the $L_\infty$-gain of $K$ being denoted $M^K_\infty$, the  $L_\infty$-gain of the closed-loop being denoted $M^{cl}_\infty$ and the $L_\infty$-gain between the input and the derivative of the state of the closed-loop being denoted $M^{cld}_\infty$.

If  $M^{cld}_\infty(\sum_{j=1}^\ell  \mu_j \Vert A_j\Vert  +  \sum_{k=1}^m \nu_k\Vert B_k \Vert   M^ K_\infty  + 2  +\sum_{l=1}^L \Vert A_{-l} \Vert  )<  1$ then $ K$ BIBO-stabilizes system (\ref{eq:system1}).

\end{proposition}

\begin{proof}
For neutral systems, we deal with equations 
\begin{eqnarray*}
\dot v(t)  + +\sum_{l=1}^L A_{-l} \dot v(t-H_l) & =& Av(t)+\sum_{j=1}^\ell A_j v(t-h_j)  + B {\cal K}*v(t) + \sum_{k=1}^m B_k {\cal K}*v(t-T_k) 
\nonumber\\
&  & + +\int_0^{D} h(\theta) v(t-\theta)\, d\theta    +Br(t) + \sum_{k=1}^m B_k r(t-T_k), \qquad (t > 0),\nonumber
\end{eqnarray*}
and 
\begin{eqnarray*}
\dot y(t) + \sum_{l=1}^L  A_{-l} \dot v(t-H_l)    & = & Av(yt)+\sum_{j=1}^\ell A_j y(t-h_j)   +\sum_{j=1}^\ell A_j f_j(t) \nonumber\\
& & + \sum_{l=1}^L  A_{-l}  (x(t-H_l)-x(t-\gamma(t)) + +\int_0^{D} h(\theta) y(t-\theta)\, d\theta  \\
& & + \int_D^{\delta(t)} h(\theta)x(t-\theta) \, d\theta \\
&  & + B (K*y)(t) + \sum_{k=1}^m B_k ({\cal K}*y)(t-T_k) +Br(t)+ \sum_{k=1}^m B_k r(t-T_k) \nonumber  \\
&   &+ \sum_{k=1}^m B_k ({\cal K}*\tilde f _k) +\sum_{k=1}^m B_k {\cal K}*(r(t-\sigma_k(t))-r(t-T_k)) , \qquad (t > 0),
\end{eqnarray*}
from which we get
%\red{
%\beq
\begin{multline*}
\Vert \dot x \Vert_\infty \leq \Vert \dot v \Vert_\infty  + M_\infty ^{cld}( 2    \sum_{l=1}^L   \Vert A_{-l}\Vert \Vert \dot x \Vert_\infty \sum_{j=1}^\ell \Vert A_j\Vert  \mu_j \Vert \dot x \Vert_\infty  \\
+ \sum_{k=1}^m \Vert B_k \Vert \nu_k  M_\infty ^K  \Vert \dot x\Vert  + 2 \sum_{k=1}^m \Vert B_k \Vert M_\infty ^K \Vert r \Vert_\infty + \epsilon \Vert h \Vert _\infty\Vert x \Vert _\infty) .
\end{multline*}
%\eeq}

\end{proof}

Although the existence of a stable controller cannot be guaranteed in the general case, let us mention that the existence of a stabilizing controller can be indeed guaranteed in the particular case of systems with commensurate delays.

It has been shown in \cite{BFP11} that neutral systems with commensurate delays and a finite number of poles in $\{ \Re s > a\}$ with $a<0$ are stabilizable in an $H_\infty$-sense:  coprime factorizations over $H_\infty$ have been  determined and the set of all stabilizing controllers was given. It is not difficult to see that the coprime factorizations are also in $\hat{\cal A}$, inducing BIBO-stabilizability as well. 

The next Proposition shows that, under hypothesis (H) systems (\ref{eq:system2}) with commensurate delays fall into the study of \cite{BFP11}.

\begin{proposition}
If the system (\ref{eq:system2}) has commensurate delays, \red{then} there exists $a<0$ such that system  (\ref{eq:system2}) only has a finite number of poles in $\{ \Re s > a\}$. 
\end{proposition}

\begin{proof}
It is well-known \cite{BC} that the location of chains of poles of system (\ref{eq:system2}) can be determined from the denominator of its transfer function, more precisely from the coefficient (containing exponential terms) of the term $s^n$. 

Considering $\det(sI-A+\sum_{\ell=1}^L A_{\ell}s e^{-s \ell H} - \sum_{j=1}^J A_j e^{-sh_j} - \int_0^D h(\theta) I e^{-\theta s}d\theta)$ we notice that it is sufficient to look at $\det(sI+\sum_{\ell=1}^L A_{-\ell}se^{-s \ell})$ as other terms do not contribute to the coefficient of the 
term $s^n$. 

Letting $z=e^{-s\ell}$, we get $\det(sI+\sum_{\ell=1}^L A_{-\ell}se^{-s \ell})= s^n \det(I + \sum_{\ell=1}^L A_{-\ell}z^{\ell})$.

Now, assume that there is $z \in \CC $ such that $ \vert z \vert \leq 1$  and $\det(I + \sum_{\ell=1}^L A_{-\ell}z^{\ell})=0$. Then there exits $x^* \in \CC ^n \setminus {0}$ with $\Vert x^* \Vert =1$ such that  $ x^*= -( \sum_{\ell=1}^L A_{-\ell}z^{\ell})x^*$. 

In this case, by hypothesis (H)
\begin{equation*}
1= \Vert -x^* \Vert= \Vert \sum_{\ell=1}^L A_{-\ell}z^{\ell} x^*  \Vert \leq \sum_{\ell=1}^L \Vert A_{-\ell} \Vert \Vert x^* \Vert\vert z \vert  \leq ( \sum_{\ell=1}^L \Vert A_{-\ell} \Vert ) \vert z \vert <1
\end{equation*}
which is absurd so all roots of  $\det(I + \sum_{\ell=1}^L A_{-\ell}z^{\ell})=0$ are of modulus strictly greater than one entailing that all roots in $s$ are strictly in the left half plane. All chains of poles of neutral system (\ref{eq:system2}) are then asymptotic to vertical axes located in the open right half-plane ensuring a finite number of pole in  $\{ \Re s > a\}$ with $a<0$.
%% 
%%The roots of $\det(sI+\sum_{\ell=1}^L A_{-\ell}se^{-s\ell})$ correspond to the position of vertical lines in the complex plane which are asymptots 
 \end{proof}

%These results remain valid in the case where the delays of the retarded terms of the transfer functions considered in \cite{BFP11} are not commensurate.

\section{Example}\label{sec:ex}

Consider the elementary delay system
\[
\dot x(t) + x(t-h) = u(t),
\]
 with transfer function
$G_h(s)=1/(s+e^{-sh})$. This is $H_\infty$ and BIBO stable provided that $0 \le h < \pi/2$ (see, for example, \cite[Chap. 6]{Par04}).

Now we consider the perturbed system
\beq\label{eq:pertdelay}
\dot x(t)+ x(t-\tau(t)) = u(t),
\eeq
with $0 \le h < \tau (t) \le h+\mu$.

By Theorem \ref{thm:2.1}, we have $H^\infty$ stability if $\mu < M^d_2{}^{-1}$.

For $h=0,0.5, 1$ and $ 1.5$ these values are
$1$, $0.63$, $0.32$ and $ 0.03$ respectively; naturally $h+\mu < \pi/2$ in all cases. 

For BIBO stability a similar result holds, except that we require the BIBO norm of $sG_h(s)=1-e^{-sh}G_h(s)$, which is not easy to calculate as we do not have an explicit form of the impulse response. One way round this is to use the Hardy--Littlewood inequality given in \cite[p. 182]{GLS} (see also \cite{BP02}), namely that
\[
\|G_h\|_{BIBO} \le \frac12 \|G'_h\|_{L^1(i\RR)}.
\]
Using this bound, we find   that for $h=0,0.5,1$ and $1.5$ the BIBO norm of $G_h$ is at most $1$, $1.01$, $2.96$ and $39.1$ respectively, giving BIBO
stability for $\mu$ at most $0.5$, $0.50$, $0.25$ and $0.025$ respectively.\\

Now for $h=2$ the system $G_h$ is not stable, but it is easily stabilized with the constant controller $K=-1$, giving a closed-loop transfer function of
\[
G_{cl}(s)=\frac{s+e^{-2s}}{s+1+e^{-2s}}.
\]
Calculations indicate that $\|G_{cl}\|_\infty=1.54$ and $\|G_{cl}\|_{BIBO} \le 3.89$. 
We may therefore apply Theorem \ref{thm:4.1}, and conclude that $K$ stabilizes the 
system (\ref{eq:pertdelay}) provided that $2 \le \tau(t) \le 2+\mu$, where $\mu=0.26$ 
in the BIBO case and $\mu=0.65$ in the $H_\infty$ case.\\

\section*{Acknowledgements}

We are grateful to Asmahan Alajyan for useful discussions of the numerical results, \red{A. Fioravanti for recalling to us a crucial fact in the proof of Theorem 4.2. and M. Souza for his help in making our proof more direct}.

\end{document}